\documentclass{proc-l}
\usepackage{graphicx}
\usepackage{amssymb,verbatim}

\usepackage[onehalfspacing]{setspace}

\usepackage{prettyref}

\newtheorem{thm}{Theorem}
\newtheorem{lem}[thm]{Lemma}

\newtheorem{cor}[thm]{Corollary}
\theoremstyle{definition}

\newtheorem{defn}[thm]{Definition}
\newrefformat{defn}{Definition \ref{#1}}
\newtheorem{remark}[thm]{Remark}
\newtheorem{exm}[thm]{Example}

\newcommand{\nats}{ {\mathbb N} }
\newcommand{\complex}{ {\mathbb C} }
\newcommand{\Disk}{ {\mathbb D} }

\newcommand{\fs}{\ensuremath{\stackrel{\mathrm{FS}}{\sim}}}

\newcommand{\ol}{\overline}
\newcommand{\sm}{\setminus}

\newcommand{\usc}{upper semi-continuous}
\newcommand{\FS}{finitely Suslinian}
\newcommand{\LC}{locally connected}

\newcommand{\A}{\mathcal{A}}

\newcommand{\B}{\mathcal{B}}

\newcommand{\M}{\mathcal{M}}
\newcommand{\fm}{\mathcal{FM}}

\newcommand{\C}{\mathbb{C}}

\newcommand{\ucirc}{\mathbb{S}^1}
\newcommand{\uc}{\mathbb{S}^1}

\newcommand{\bd}{\mbox{Bd}}

\newcommand{\e}{\varepsilon}
\newcommand{\al}{\alpha}

\newcommand{\ph}{\varphi}
\newcommand{\be}{\beta}

\newcommand{\vp}{\varphi}

\newcommand{\imp}{\mbox{Imp}}

\newcommand{\disk}{\mathbb{D}}

\newcommand{\iy}{\infty}

\usepackage[usenames]{color}
\usepackage{verbatim}

\begin{document}
\title[Finitely Suslinian models for planar compacta]
{Finitely Suslinian models for planar compacta with applications to Julia sets}
\author{Alexander Blokh}
\thanks{The first author was partially
supported by NSF grant DMS-0901038.}
\address[A.~Blokh] {Department of Mathematics\\ University of Alabama at
Birmingham\\ Birmingham, AL 35294-1170}
\email[A.~Blokh]{ablokh@math.uab.edu}
\author{Clinton Curry}
\thanks{The second author was partially supported by NSF grant DMS-0353825.}
\address[C.~Curry] 
{Institute for Mathematical Sciences\\
Stony Brook University\\
Stony Brook, NY 11794}
\email[C.~Curry]{clintonc@math.sunysb.edu}
\author{Lex Oversteegen}
\thanks{The third author was partially  supported
by NSF grant DMS-0906316.}
\address[L.~Oversteegen]{Department of Mathematics\\ University of Alabama
at Birmingham\\Birmingham, AL 35294-1170}
\email[L.~Oversteegen]{overstee@math.uab.edu}
\date{September 7, 2010; revised August 22, 2011}



\begin{abstract}
  A compactum $X\subset \C$ is \textbf{unshielded} if it
  coincides with the boundary of the unbounded component of $\C\sm
  X$. Call a compactum $X$ \textbf{\FS{}} if every
  collection of pairwise disjoint subcontinua of $X$ whose diameters
  are bounded away from zero is finite.  We show that any unshielded
  planar compactum $X$ admits a topologically unique monotone map
  $m_X:X \to X_{FS}$ onto a \FS{} quotient such that any monotone
  map of $X$ onto a \FS{} quotient factors through $m_X$.  We call
  the pair $(X_{FS},m_X)$ (or, more loosely, $X_{FS}$) the
  \textbf{finest \FS{} model of $X$}.

  If $f:\C\to \C$ is a branched covering map and $X \subset \C$ is a
  fully invariant compactum, then the appropriate extension $M_X$ of
  $m_X$ monotonically semiconjugates $f$ to a branched covering map
  $g:\C\to \C$ which serves as a model for $f$. If $f$ is a
  polynomial and $J_f$ is its Julia set, we show that $m_X$ (or
  $M_X$) can be defined on each component $Z$ of $J_f$ individually
  as the finest monotone map of $Z$ onto a locally connected
  continuum.
\end{abstract}

\maketitle

\section{Introduction}

For us, a \textbf{compactum} is a non-empty compact metric space. A
compactum is \textbf{degenerate} if all of its components are points. A
\textbf{continuum} is a connected compactum. One way of describing the
topology of a compactum $X$ is by constructing a \emph{model} for it,
i.e. a compactum $Y$, simpler to describe than $X$, and a (monotone)
onto map $m:X\to Y$ (a continuous onto map $m$ is \textbf{monotone} if
all $m$-preimages of continua are continua; we denote the family of all
monotone maps by $\M$). If $X$ carries an additional structure, it is
nice if the map $m$ preserves that structure (e.g., if there is a continuous map
$f:X \to X$, the map $m$ should be chosen so that $f$ induces a
continuous self-map on $Y$ by $m(x) \mapsto m(f(x))$). In this case $m$ is said
to be a \textbf{monotone semiconjugacy} of the map $f:X \to X$ to the
induced map $g:Y \to Y$ (if $m$ is a homeomorphism, it is called a
\textbf{conjugacy}).

Unless specified otherwise, \emph{from now on all compacta we
consider are planar}. A case of particular interest is when $X$ is
\textbf{unshielded}, i.e. $X\subset \C$ is the boundary of the
unbounded component $U_\infty(X)$ of $\C \setminus X$.  The
following construction is due to Carath\'eodory.  Recall that a
space is \textbf{locally connected} if its topology has a basis of
connected sets. If $X$ is an unshielded continuum, then
$U_\infty(X)\cup \{\infty\}$ is a simply connected open set in the
Riemann sphere. Take the unique Riemann map $\varphi_X:\disk \to
U_\infty(X)$ with positive derivative at the origin (here $\disk$ is
the unit open disk centered at the origin). If $X$ is locally
connected, we may extend $\varphi_X$ continuously to
$\overline{\disk}$, mapping $\ucirc$ onto $X$ (here $\ucirc$ is the
boundary of $\disk$). Declaring points $u, v\in \uc$ equivalent if
and only if $\vp_X(u)=\vp_X(v)$ and denoting this equivalence
relation by $\approx$, we see that $X$ is homeomorphic to the
quotient space $\uc/\approx$.  Equivalence relations $\approx$ which
arise in this way are called \textbf{laminations}.  If $J_P$ is the
locally connected Julia set of a polynomial $P$, then $\varphi_X$
semiconjugates $z^d|_{\ucirc}$ to a map $\widetilde{P}:\uc/\approx
\to \uc/\approx$.

A lamination can be defined in abstract circumstances as a closed
equivalence relation $\approx$ on $\ucirc$ such that convex hulls of
$\approx$-classes are pairwise disjoint (here the \emph{convex hull}
of a set $T$ is the smallest convex set containing $T$). Laminations
therefore capture the external ray picture of unshielded continua.
In order to model dynamical objects like the Julia set of a degree
$d$ polynomial, we may require that $\approx$ is
\textbf{$d$-invariant}. This means that the image of a
$\approx$-class under the angle $d$-tupling map is again a
$\approx$-class, and classes map to each other in a
consecutive-preserving way (loosely speaking, preserving the order
of points on the circle).

There are even laminations for disconnected Julia sets; here
$\approx$ is a closed equivalence relation defined on a Cantor
subset $A \subset \ucirc$, and the angle $d$-tupling map is replaced
by a covering self-map of $A$.  This models that, for a polynomial
$P$ with disconnected Julia set $J_P$, the neighborhood of $\infty$
on which $P$ is conjugate to $z \mapsto z^d$ does not include the
entire basin of infinity.  In this case every external ray can be
analytically continued until it runs into the Julia set unless it
first runs into the preimage of an escaping critical point. In such
a case, one can take left- and right-sided limits of fully-defined
external rays and define two external rays corresponding to the same
angle. These angles are associated to (pre)critical points and to
the gaps in the Cantor set $A$ (see \cite{gm93, kiw04, lp96}).

By Kiwi \cite{kiw04} laminations correspond to a wider class of
polynomials $P$, whose Julia sets may not be locally connected nor
connected. More precisely, an $n$-periodic point $a$ of $P$ is
called \textbf{irrationally neutral} if $(P^n)'(a)=e^{2\pi i\al}$
with $\al$ irrational. Also, given a lamination $\approx$ of $\uc$,
call a set $F\subset \uc$ \textbf{$\approx$-saturated} if it is a
union of a collection of $\approx$-classes. By \cite{kiw04}, to
every polynomial $P$ without irrationally neutral cycles we can
associate a lamination $\approx$, a closed $\approx$-saturated set
$F\subset \uc$ and a monotone map $m:J_P\to F/\approx$ such that $m$
is a semiconjugacy of $P|_{J_P}$ with an appropriately constructed
map $f:F/\approx \to F/\approx$ (in the case that $J_P$ is
connected, then $F=\uc$ and $f$ is a map induced on $\uc/\approx$ by
$z^d$).

The present authors prove \cite{bco08} that every complex polynomial
$P$ with \emph{connected} Julia set has a unique ``best'' lamination.
This generalizes \cite{kiw04}, albeit for connected Julia sets, by
allowing $P$ to have irrationally neutral cycles. The lamination
$\approx$ comes with a monotone semiconjugacy $m:J_P \to \ucirc /
\approx$ which has the property of being the \emph{finest monotone map
of $J_P$ onto a locally connected continuum} (defined in the next
section). In \cite{bco08} we also provide a criterion for $\approx$ to
have more than one equivalence class (equivalently, for $J_P$ to have a
non-degenerate locally connected monotone image).

A compactum $X$ is called \textbf{finitely Suslinian} if, for every
$\e > 0$, every collection of disjoint subcontinua of $X$ with
diameters at least $\e$ is finite. By Lemma 2.9 \cite{bo04},
\emph{unshielded planar \LC~ continua are \FS{} and vice
  versa}\footnote{If $X$ is not unshielded, this may fail as the
  closed unit disk $\overline{\disk}$ is \LC{} but not \FS;
  Example~\ref{lcnotfs} shows that there are nowhere dense \LC{}
  planar continua which are not \FS.}. Thus, in the unshielded case
the notion of finitely Suslinian generalizes the notion of local
connectivity. There is another analogy to local connectivity too: by
Theorem 1.4 \cite{bmo07}, \emph{for an unshielded \FS{} compactum
  $X\subset \C$ there exists a lamination $\approx$ of a closed set $F
  \subset \uc$ such that $X$ is homeomorphic to $F / \approx$}. This
motivates us to extend onto \FS{} compacta some results for \LC{}
continua and to look for good \FS{} models of planar compacta. We need
the following definition which applies to arbitrary maps (as customary
in topology, by a map we always mean a \emph{continuous} map).

\begin{defn}[Finest models]\label{def:finmod}
  Let $X \subset \complex$ be a compactum, P be a topological property
  (P could be the property of being locally connected, Hausdorff, etc)
  and $\B$ be a class of maps with domain $X$. The \textbf{finest
    $\B$-model of $X$ with property P} is an onto map $\psi^P:X \to
  Y,\ \psi^P\in \B$ where $Y$ is a topological space with property
  $P$ such that any other map $\vp^P:X \to Z,\ \vp^P\in \B$ onto a
  space $Z$ with property $P$ can be written as the composition $g
  \circ \psi^P$ for some map $g:Y \to Z$.
\end{defn}

Though we give the definition for any class $\B$, we are mostly
interested in the class $\M$ of monotone maps because such maps do not
change the structure of $X$ too drastically; besides, we study planar
compacta, and monotone maps of planar compacta with non-separating
fibers keep them planar \cite{moo62}. In the monotone case we will use
notation $m^P$ instead of $\psi^P$ (or just $m$ if the property $P$ is
fixed). In fact, in the monotone case this concept of finest map has
been studied before in the context of continua (cf
\cite{Swingle:Core}).

\begin{lem}\label{finunik}
If the finest $\B$-model with property $P$ exists, then it is unique up
to a homeomorphism.
\end{lem}

\begin{proof}
Suppose that $m_1:X \to Y_1$ and $m_2:X \to Y_2$ are finest
$\B$-models. Then, by definition, we may factor $m_1$ as
\[ m_1:X \stackrel{m_2}\to Y_2 \stackrel{g_1}{\to} Y_1 \] and
similarly factor the constituent map $m_2$ to obtain \[m_1:X
\stackrel{m_1}{\to} Y_1 \stackrel{g_2}\to Y_2 \stackrel{g_1}{\to}
Y_1.\] However, since the composition is itself equal to $m_1$, we
find that $g_1:Y_2 \to Y_1$ and $g_2:Y_1 \to Y_2$ are each other's
inverse and hence homeomorphisms.  Therefore, $Y_2 = g_2(Y_1)$ is
homeomorphic to $Y_1$, $m_2 = g_2 \circ m_1$, and $m_1 = g_1 \circ
m_2$.
\end{proof}

The following notion is a bit weaker than that defined in Definition~\ref{def:finmod}.

\begin{defn}[Top models]\label{def:topmod}
Let $h:X\to Y$ be a map in $\B$ onto a compactum $Y$ with property P
such that there exists no map $h':X\to Y'$ onto a compactum $Y'$ with
property P which refines $h$ (i.e., if a map $h''$ is such that
$h=h''\circ h'$, then $h''$ must be a homeomorphism and $Y'$ is
homeomorphic to $Y$). Then $(Y, h)$ is said to be a \textbf{top
  $\B$-model of $X$ with property P}.
\end{defn}

Observe that, while the finest model is finer than
all others, a top model does not have another strictly finer
model. This is the same as the greatest model and a maximal model in the sense of
some partial order.
Hence, if the finest $\B$-model of $X$ with property P exists,
it is the unique top $\B$-model of $X$ with property P. So, if we have
a top $\B$-model $h:X\to Y$ of $X$ with property P and a $\B$-model
$h':X\to Y'$ of $X$ with property P such that $h$ is not finer than
$h'$, then the finest $\B$-model of $X$ with property P does not
exist.  Example~\ref{lcnotfs} provides a planar continuum $X_1$ with
this situation in the case when P is the property of being finitely
Suslinian and $\B$ is either the class of continuous maps or monotone
maps; thus, $X_1$ has no finest continuous or monotone model with
\FS{} property.


Also, by the definitions if $\B\subset \B'$ are two
classes of maps and the finest (top) $\B'$-model of $X$ with property
P is $(Y, h)$ where $h$ happens to belong to the smaller class $\B$,
then $(Y, h)$ is also the finest (top) $\B$-model of $X$ with property P.

The purpose of the paper is to prove Theorems~\ref{thm:finmod},
~\ref{thm:branchcover} and ~\ref{thm:julifs}. As
Theorem~\ref{thm:finmod} proves, the situation with
\emph{unshielded} planar continua and the \FS{} property is better.
From now on finest $\M$-models with \FS{} property will be called
\emph{finest \FS{} monotone models.}

\begin{thm}\label{thm:finmod}
Every unshielded compactum $X$ has a finest \FS{} monotone model
$m_X:X\to X_{FS}$.
\end{thm}

This yields applications to the dynamics of branched covering maps of
the plane, and in particular the study of Julia sets of polynomials,
which are naturally occurring examples of unshielded compacta.

\begin{thm}\label{thm:branchcover}
  Suppose that $f:\C\to \C$ is a branched covering map and $X$ is an
  unshielded compactum which is fully invariant under $f$. Then
  $X_{FS}$ can be embedded into the plane and the finest \FS{}
  monotone model $m_X:X\to X_{FS}$ can be extended to the plane in
  such a way that the resulting map $M_X:\C\to \C$ semiconjugates
  $f$ and a branched covering map $g:\C \to \C$.
\end{thm}

These results can be made stronger if $f$ is a polynomial.

\begin{thm}\label{thm:julifs}
  The finest \FS{} monotone model $m_{J_P}:J_P \to J_{P_{FS}}$ of the Julia
  set of a polynomial $P$ coincides on each component $X$ of $J_P$
  with the finest monotone map $m_X$ of $X$ to a \FS{} continuum.  In
  particular:
  \begin{enumerate}
  \item the finest \FS{} monotone model of $J_P$ is non-degenerate if and only
    if there exists a periodic component of $J_P$ whose finest \FS{} monotone
    model is non-degenerate;
  \item the set $J_P$ is \FS{} if and only if all periodic
    non-degenerate components of $J_P$ are locally connected.
  \end{enumerate}
\end{thm}

By \cite{bco08}, one can specify exactly the situations in which a
non-degenerate finitely Suslinian model of a polynomial Julia set
exists. This is because any periodic component of $J_P$ is the Julia
set of a polynomial-like map, which is hybrid equivalent (in
particular, topologically conjugate) to a polynomial.  Hence,
summarizing the results of \cite{bco08}, we conclude that a periodic
component $Y$ of $J_P$ has a non-degenerate \FS{} model if and only if
one of the following is true:

\begin{enumerate}
\item $Y$ contains infinitely many periodic points, each of which separates $Y$,
\item the topological hull of $Y$ contains either a parabolic or
  attracting periodic point, or
\item $Y$ admits a \emph{Siegel configuration}, which roughly means
  that are subcontinua of the Julia set, comprised of finitely many
  impressions and disjoint from all other impressions, which in
  essence correspond to the critical points on the boundaries of
  Siegel disks in locally connected Julia sets.
\end{enumerate}

For all details, the reader is invited to read \cite{bco08},
especially Section 5 thereof.

\section{Topological Lemmas}
\label{sec:top-lemmas}

First we introduce several useful notions. When speaking of limits of
compacta, we always mean convergence in the Hausdorff sense.

\begin{defn}\label{usc}
  A partition of a compactum $X$ if said to be \textbf{\usc{}} if for
  every pair of convergent sequences $(x_i)_{i=1}^\infty$ and
  $(y_i)_{i=1}^\infty$ of points in $X$ such that $x_i, y_i$ belong to
  some element $D_i$ of the partition, we have that the points
  $\lim_{i \to \infty} x_i$ and $\lim_{i \to \infty} y_i$ belong to
  some element $D$ of the partition. In this case the equivalence
  relation $\sim$ induced by the partition is said to be
  \textbf{closed}.  Equivalently, $\sim$ is said to be \textbf{closed}
  if its graph is closed in $X\times X$.
\end{defn}

The following construction is less standard.

\begin{defn}\label{ra}
  Let $\A$ be a family of 
  subsets of a compactum $X$.  An equivalence relation $\sim$
  \textbf{respects $\A$} if $\sim$ is closed and
  every member of $\A$ is contained in a $\sim$-class.  If $\sim$
  and $\approx$ are equivalence relations on a set $X$, we say that
  $\sim$ is \textbf{finer} than $\approx$ if $\sim$-classes are
  contained in $\approx$-classes.  The \textbf{finest closed
  equivalence relation generated by $\A$} is the finest equivalence
  relation $\sim_\A$ respecting $\A$.

Equivalently, one can define \textbf{continuous maps respecting
$\A$} as maps which collapse all elements of $\A$ to points. Then we
can define the \textbf{finest continuous map respecting $\A$}, i.e.
a continuous map $\psi^\A:X\to Z$ respecting $\A$ and such that for
any map $f:X\to Q$ which respects $\A$ there exists a map $g:Z\to Q$
which can be composed with $\psi^\A$ to give $f=g\circ \psi^\A$.
\end{defn}

Lemma~\ref{lem:connected_classes-1} shows that the finest closed
equivalence relation generated by $\A$ (and hence, the finest map
respecting $\A$) exists and specifies its properties if elements of
$\A$ are connected.

\begin{lem}\label{lem:connected_classes-1}
  The finest closed equivalence relation generated by $\A$ exists
  and is therefore unique (thus, the finest map $\psi^\A$ respecting
  $\A$ exists and is well-defined). If $\A$ consists of connected
  subsets of a compactum $X$, then all $\sim_\A$-classes are
  continua and the finest continuous map respecting $\A$ is
  monotone.
\end{lem}

\begin{proof}
To see that $\sim_\A$ is well-defined, let $\Xi_\A$ be the set of
all upper semi-continuous equivalence relations which respect $\A$
($\Xi_\A$ is non-empty as it includes the trivial equivalence
relation under which all points are equivalent). Then it is easy to
see that the relation $\sim_\A$ defined by ``$x \sim_\A y$ if and
only if $x \sim y$ for all $\sim \in \Xi_\A$'' is again a closed
equivalence relation respecting $\A$, and that $\sim_\A$ is finer
than all closed equivalence relations from $\Xi_\A$. It follows that
the quotient map $X\to X/\sim_\A$ is in fact the finest continuous
map which respects $\A$.

  It suffices to show that all $\sim_\A$ classes are connected.
  According to \cite[Lemma 13.2]{Nadler:ContTh}, the equivalence
  relation $\sim$ whose classes are the components of
  $\sim_\A$-classes is also an \usc{} equivalence relation, and
  $\sim$-classes are contained in $\sim_\A$-classes. Since elements of
  $\A$ are connected, it follows that $\sim$ still respects $\A$, so
  $\sim_\A$ classes are contained in $\sim$-classes.  Therefore $\sim
  = \sim_\A$, and $\sim_\A$-classes are connected.
\end{proof}

It is quite easy to determine when a continuous function on $X$
induces a continuous function on $X / \sim_\A$, as the following lemma
shows.

\begin{lem}\label{lem:induced}
  If $f:X \to X$ is a continuous function which sends elements of $\A$
  into $\sim_\A$-classes, then $f$ induces a function $g:X / \sim_\A
  \to X / \sim_\A$ with $\psi^\A \circ f=g\circ \psi^\A$ ($g$ maps the
  $\sim_\A$-class of $x$ to the $\sim_\A$-class of $f(x)$).
\end{lem}

\begin{proof}
  It is sufficient to show that the $f$-image of a $\sim_\A$-class
  is contained in a $\sim_\A$-class.  Consider the fibers of $\psi^\A
  \circ f$. By assumption, $f$ sends elements of $\A$ into
  $\sim_\A$-classes, so $\psi^\A \circ f$ is constant on the elements
  of $\A$.  Therefore, the fibers of $\psi^\A \circ f$ form an
  \usc{} partition of $X$ which respects $\A$.  Since $\psi^\A$ is the
  finest such map in the sense of Definition~\ref{ra},
  there exists a map $g: X / \sim_\A \to X / \sim_\A$ with
  $\psi^\A \circ f=g\circ \psi^\A$ as desired.
\end{proof}

\begin{remark}\label{rem:transfinite}
  For later reference, we note that there is also a transfinite
  construction of the equivalence relation $\sim_{\mathcal A}$.  To
  begin, let $\sim_0$ denote the equivalence relation such that $x
  \sim_0 y$ if and only if $x$ and $y$ are contained in a connected
  finite union of elements of $\A$. If an ordinal $\alpha$ has an
  immediate predecessor $\beta$ for which $\sim_{\beta}$ is defined,
  we define $x \sim_\alpha y$ if there exist finitely many sequences
  of $\sim_{\beta}$ classes whose limits comprise a continuum
  containing $x$ and $y$ (here, the limit of non-closed sets is
  considered to be the same as the limit of their closures).  In the
  case that $\alpha$ is a limit ordinal, we say $x\sim_\alpha y$
  whenever there exists $\beta < \alpha$ such that $x \sim_\beta y$.
  Notice that the sequence of $\sim_\alpha$-classes of a point $x$ (as
  $\alpha$ increases) is an increasing nest of connected sets, with
  the closure of each being a subcontinuum of its successor.  It is
  also apparent that $\sim_\alpha$-classes are contained in
  $\sim_\A$-classes for all ordinals $\alpha$.

  Let us now show that $\sim_{\mathcal A} = \sim_\Omega$ where $\Omega$
  is the smallest uncountable ordinal.  To see this, we first note that
  $\sim_\Omega = \sim_{(\Omega+1)}$.  This is because the sequence of
  closures of $\sim_\alpha$-classes containing a point $x$ forms an
  increasing nest of subsets, no uncountable subchain of which can be
  strictly increasing in the plane \cite[Theorem 3, p.~258]{kur66}.
    Therefore, all
  $\sim_\alpha$-classes have stabilized when $\alpha = \Omega$.  This
  implies that $\sim_\Omega$ is a closed equivalence relation, since
  the limit of $\sim_\Omega$-classes is a $\sim_{(\Omega+1)}$-class,
  which we have shown is a $\sim_\Omega$-class again.  Finally,
  $\sim_\Omega$ respects $\A$ and $\sim_\Omega$-classes are contained
  in $\sim_{\mathcal A}$-classes, so $\sim_{\mathcal A}$ and
  $\sim_\Omega$ coincide.
\end{remark}

Let us become more specific and study \FS{} compacta.

\begin{defn}[Limit continuum and $\fs$]\label{rx}
  A subcontinuum $C$ of $X$ is said to be a \textbf{limit continuum} if
  there exists a sequence $(C_n)_{n=1}^\infty$ of pairwise disjoint
  subcontinua of $X$ converging to $C$.  We define $\fs_X$ as the
  \textbf{finest equivalence relation respecting the family of limit
    continua} (if the context is clear, we may omit the subscript
  and refer simply to $\fs$).
\end{defn}

Note that this notion is slightly more general than the classical
notion of \emph{continuum of convergence} in continuum theory.  Also,
it is easy to see that a continuum is \FS{} if and only if it contains
no non-degenerate limit continua.

\begin{lem}\label{lem:ffs}
  For any compactum $X$, the quotient $X / \fs$ is \FS.
\end{lem}

\begin{proof}
  Let $(C_n)_{n=1}^\infty$ be a (without loss of generality
  convergent) sequence of pairwise disjoint subcontinua of $X /
  \fs$. Let $m:X \to X / \fs$ denote the quotient map.  A
  subsequence of the preimages $(m^{-1}(C_n))_{n=1}^\infty$
  converges to a continuum $K$. By definition of $\fs$ we have that
  $m(K)$ is a singleton, say, $\{a\}$, and continuity of $m$ implies
  that $(C_n)_{n=1}^\infty$ converges to $\{a\}$. Since
  $(C_n)_{n=1}^\infty$ was arbitrary, we have that $X/ \fs$ contains
  no non-degenerate limit continua and is therefore \FS{}.
\end{proof}

\prettyref{lem:ffs}, together with the characterization of \FS{} compacta as
those with no limit continua, suggests that $X / \fs$ could be the finest model
of $X$. Such a fact would mean that any monotone map of $X$ onto a \FS{}
compactum must collapse limit continua. However, in general this is not true.

\begin{exm}[A continuum with no finest \FS{} model]\label{lcnotfs}
  De\-fine a continuum $X$ as follows, and as depicted in
  \prettyref{fig:nofs}:
  \begin{align*}
    H_n &= [0,1] \times \{1/2^n\},\,n\in\nats,\\
    H & = [0, 1]\times \{0\},\\
    V_{p/q} &= \{p/q\} \times [0,1/q] \text{, $p/q$ a dyadic rational}
  \end{align*}
  \[ X_1 = \bigcup\{H_n\,\mid\, n \in \nats\} \cup H \cup
  \bigcup\{V_{p/q}\,\mid\,0<p/q<1\text{ dyadic}\}
  \]
  \begin{figure}
    \includegraphics[width=.4\textwidth]{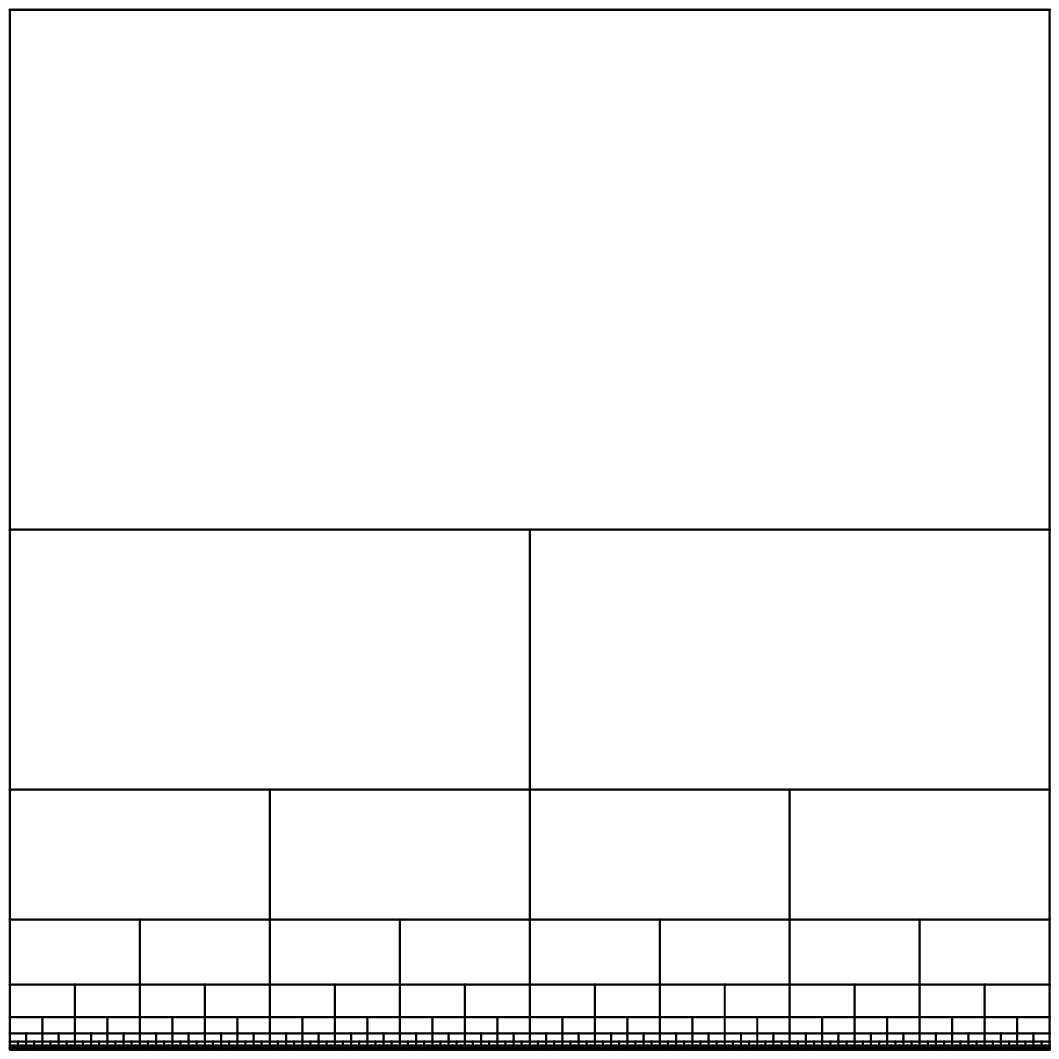}
    \caption{A continuum with no finest \FS{} model.}
    \label{fig:nofs}
  \end{figure}
  Observe that $X_1$ is a locally connected, not \FS{}, nowhere dense
  and \emph{not} unshielded in $\C$ continuum. There are two
  essentially different kinds of \FS{} monotone quotients of $X_1$,
  depicted in Figure~\ref{fig:fsquotients}. One map, $h$, corresponds
  to identifying the unique maximal limit continuum $H = \lim_{n \to
    \infty} H_n$ to a point. Any finer (and not even necessarily
  monotone) map $h'$ to a \FS{} compactum would still keep images of
  $H_n$ disjoint, implying that images of $H_n$ must converge to a
  point which has to be the image of $H$. Thus, $h'=h$ and $(h(H), h)$
  is a top \FS{} model of $X_1$ which happens to be monotone.
  \begin{figure}
    \includegraphics[width=.4\textwidth]{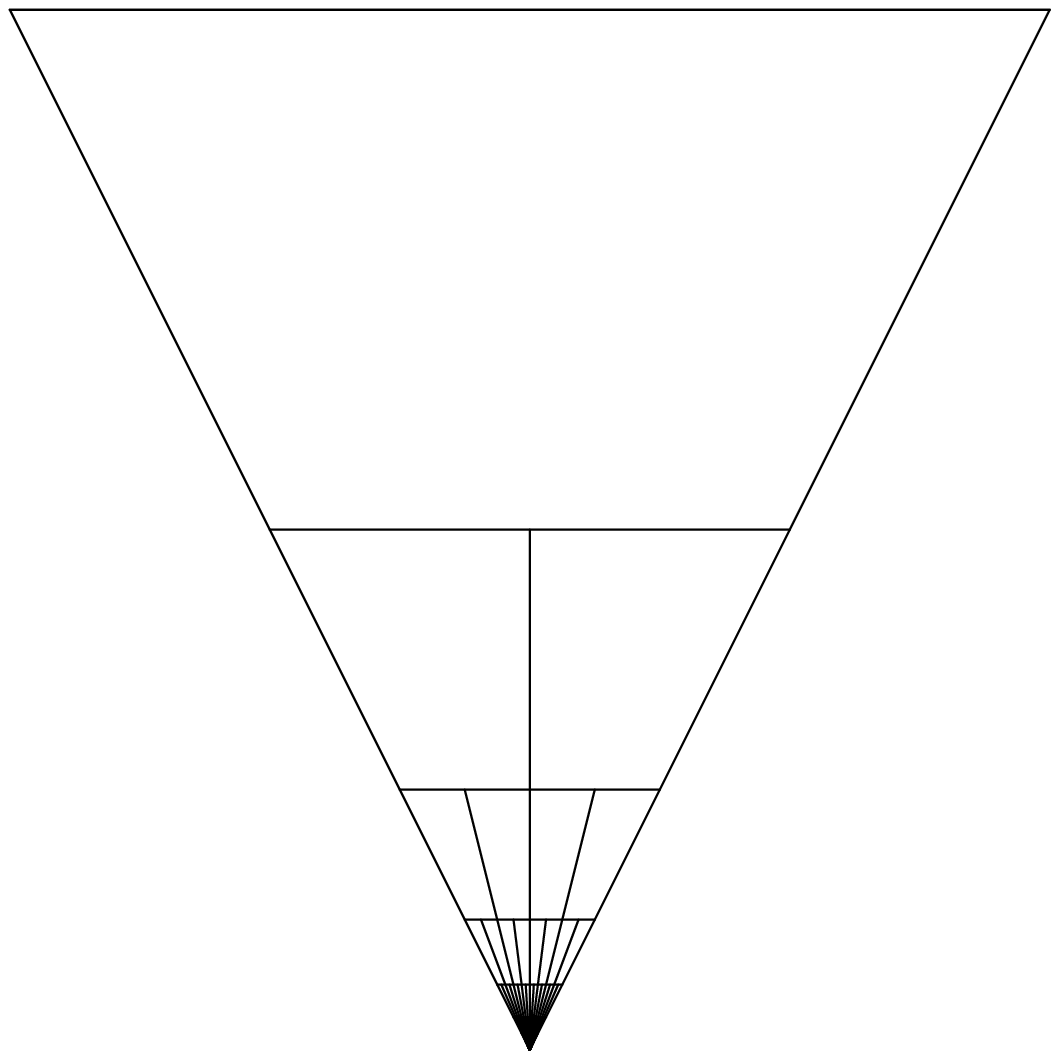}
    \includegraphics[width=.4\textwidth]{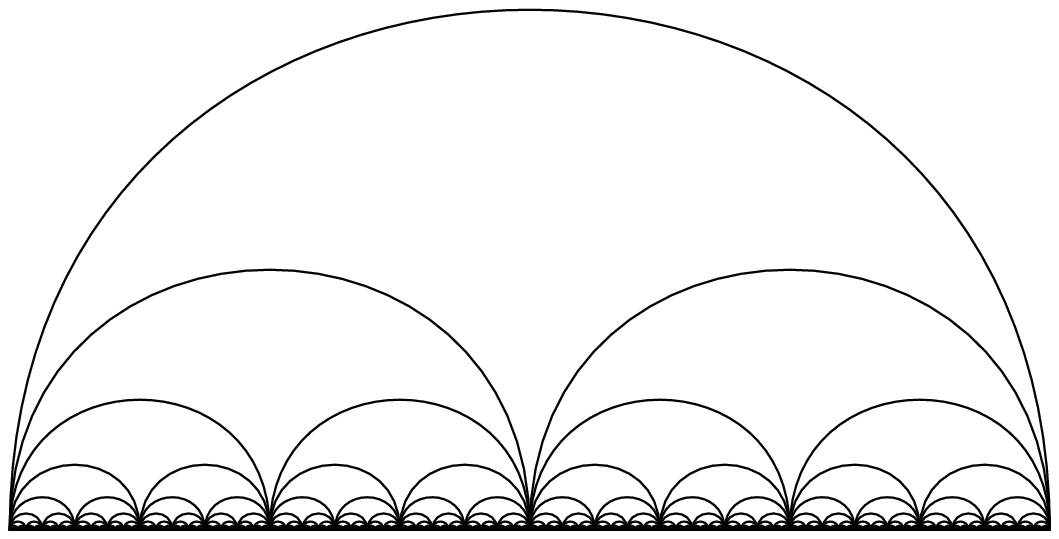}
    \caption{Essentially different \FS{} quotients of the continuum
      depicted in Figure~\ref{fig:nofs}.}
    \label{fig:fsquotients}
  \end{figure}

  Other quotients of $X_1$ with \FS{} images are maps $\ph_N$ which
  identify to points members of the collection
  $\{V_{p/q}\,\mid\,q>N\}$.  This yields a sequence of maps
  $(\ph_N)_{N=0}^\infty$, with $\ph_{N+1}$ finer than $\ph_N$ for all
  $N$. On the other hand, none of these maps can be compared with
  $\psi$ in the sense that neither $h$ is finer than $\ph_N$ nor
  $\ph_N$ is finer than $h$. As explained above, it follows from the
  definitions now that $h$ is \textbf{not} the finest \FS{} model of
  $X_1$ (neither is it the finest \FS{} monotone model of $X_1$). It is
  worth noticing also that for any $N$ the only maps finer than both
  $\ph_N$ and $h$ are homeomorphisms (since the intersection of any
  fibers of $h$ and $\ph_N$ is at most a point) and that the only
  maps finer than every map in $\{\ph_N\,\mid\,N \in \nats\}$ are
  homeomorphisms.
\end{exm}

In the unshielded case the situation is better. First we need
\prettyref{defn:irco}.

\begin{defn}[Irreducible continua]\label{defn:irco}
Given two disjoint closed sets $A, B$, a continuum $C$ is said to be
\textbf{irreducible between $A$ and $B$} if $C$ intersects both $A$ and $B$ and
does not contain a subcontinuum with the same property.  Given a continuum $D$
intersecting $A$ and $B$, one can use Zorn's Lemma to find a subcontinuum
$C\subset D$ irreducible between $A$ and $B$.
\end{defn}

We also need \prettyref{lem:irrcir}.

\begin{lem}\label{lem:irrcir} Let $K$ be an irreducible
continuum between $\partial U$ and $\partial V$ where $U, V$ are
open sets with disjoint closures. Then $K$ is disjoint from both $U$
and $V$.
\end{lem}

\begin{proof}
Set $K'=K\setminus \ol{V}$. Take a component $Y$ of $K'$ containing a point
from $\partial U$. By the Boundary Bumping Theorem (Theorem 5.6 from
\cite[Chapter V, p. 74]{Nadler:ContTh}) $\ol{Y}$ intersects $\partial V$. Since
$K$ is irreducible,  $\ol{Y}=K$ and hence $K$ is disjoint from $V$. Similarly,
$K$ is disjoint from $U$.
\end{proof}

To prove our first theorem we need the following geometric lemma.  It is a
generalization of the fact that any homeomorphic copy of the letter $\theta$
embedded in the plane is not unshielded. For a planar continuum $Y$ the set
$\C\sm U_\iy(Y)$ is called the \textbf{topological hull of $Y$} and is denoted
by $T(Y)$.

\begin{lem}\label{lem:notheta}
  Suppose a planar compactum $X$ contains two disjoint continua $X_1,$
  $X_2 \subset X$ and three pairwise disjoint continua $C^1, C^2, C^3$
  such that $X_i \cap C^j \neq \emptyset$ for all $i \in \{1,2\}$ and
  $j \in \{1,2,3\}$.  Then $X$ is not unshielded.
\end{lem}

\begin{proof}
By way of contradiction we assume that $X$ is unshielded. Let us
collapse the topological hulls $T(X_1)$ and $T(X_2)$ to points $x_1$
and $x_2$ and let $m:\C\to\C$ denote this monotone map (by Moore's
Theorem \cite{moo62}, the image is homeomorphic to the plane). Then
$m(C^i)\cap m(C^j)=\{x_1,x_2\}$ for all $i\ne j$ and $m(X)$ is also
unshielded. Put $Z^j=T(m(C^j))$. Then $Z^i\cap Z^j=\{x_1,x_2\}$ for all
$i\ne j$.  By Theorem~63.5 of \cite{mun00} for each $i\ne j$, $Z^i\cup
Z^j$ separates $\C$ into precisely two components, one of which is
bounded and denoted by $B_{i,j}$. It follows that for some choice of
$i,j,k$ the set $Z_i$ intersects $B_{j,k}$, contradicting that $m(X)$ is
unshielded.
\end{proof}




We use Lemma~\ref{lem:notheta} 
to show that certain maps of unshielded compact sets collapse limit
continua. Let $\fm$ be the class of all \textbf{(continuumwise)
finitely monotone} maps, i.e. such maps $h:X\to Y$ that for any
\emph{continuum} $Z\subset Y$ the set $h^{-1}(Z)$ consists of
finitely many components.


\begin{lem}\label{finmocol}
  Suppose that $\ph:X\to Y$ is a finitely monotone map of an
  unshielded compact set $X$ onto a \FS{} compact set $Y$.  If\,
  $C\subset X$ is a limit continuum, then $\ph(C)$ is a point.
\end{lem}

\begin{proof}
  Let $C\subset X$ be a limit continuum. Choose a sequence of continua
  $C_i\to C$. Consider two cases.

  \noindent \textbf{Case 1.} \emph{There are infinitely many distinct
    components of\, $Y$ containing sets $\ph(C_i)$.}

  Denote by $T_i$ the component of $Y$ which contains $\ph(C_i)$. We
  may refine the sequence $(C_i)_{i=1}^\infty$ so that all sets
  $T_i$ are different.  Since $Y$ is \FS, we may refine it further
  so that $T_i$ converge to a point $t\in Y$. Hence $\ph(C)=\lim
  \ph(C_i)=\lim T_i=t$.

  \noindent \textbf{Case 2.} \emph{There are finitely many distinct
    components of $Y$ containing all sets $\ph(C_i)$.}

  Since $\ph$ is finitely monotone, we may assume that all $C_n$ are
  contained in a single component $T$ of $X$. Observe that then
  $\ph(T)\subset Z$ where $Z$ is a component of $Y$. By \cite[Lemma
  2.9]{bo04}, $Z$ is \LC. We suppose that $\ph(C)$ is not a point
  and show that this contradicts the fact that $X$ is unshielded.
  Let $z_1=\ph(x_1)$ and $z_2=\ph(x_2)$ be distinct points in
  $\ph(C)$. Since $Z$ is locally connected, there exist open,
  connected subsets $Z_1, Z_2 \subset Z$ with disjoint closures,
  containing $z_i \in Z_i$ for $i \in \{1,2\}$.  Then
  $\ph^{-1}(\ol{Z_1})$ and $\ph^{-1}(\ol{Z_2})$ have finitely many
  components.  After refining the sequence $C_i$ we may assume that
  all sets $C_i$ intersect a component $A_1$ of $\ph^{-1}(\ol{Z_1})$
  and a component $A_2$ of $\ph^{-1}(\ol{Z_2})$. However, by
  \prettyref{lem:notheta} this is impossible.

\end{proof}

\section{The Existence of the Finest Map and Dynamical Applications in the Unshielded Case}
\label{sec:existence-finest-map}

\subsection{The existence of the finest map in the unshielded case}\label{sub:finmap}

We are ready to prove our first theorem which implies \prettyref{thm:finmod}.

\begin{thm}\label{finmod1}
  Let $X$ be an unshielded compact set in the plane. Then the quotient
  map $m_X:X\to X / \fs=X_{FS}$ is the finest \FS{} monotone model of
  $X$. Moreover, $X_{FS}$ can be embedded into the plane and $m_X$ can
  be extended to a monotone map $M_X:\C\to \C$ which collapses the
  topological hulls of $\fs$-classes and is one-to-one elsewhere.
\end{thm}

\begin{proof}
  By \prettyref{lem:ffs}, $X / \fs$ is a \FS{} compactum. Now, suppose
  that $\ph:X\to Z$ is monotone and $Z$ is \FS. Then $\ph$ collapses
  all limit continua by Lemma~\ref{finmocol}. Since $\fs$ is the
  finest equivalence relation respecting the collection of limit
  continua, we see that the quotient map $m_{FS}:X \to X / \fs$ is
  finer that $\ph$, and is therefore the finest \FS{} monotone model of
  $X$. The rest of the theorem follows from the Moore theorem
  \cite{moo62}.
\end{proof}

Observe that in fact Lemma~\ref{finmocol} implies that the finest
$\fm$-model of $X$ with finitely Suslinian property is the same as
the finest \FS{} monotone model of $X$ (despite the fact that the
class $\fm$ of finitely monotone maps is much wider than the class
$\M$ of monotone maps).



\subsection{Applications to Dynamical Systems}

First we show that sometimes the finest map is compatible with the
dynamics. Recall that a set $A \subset X$ is \textbf{fully
invariant} under a map $f:X \to X$ if $A = f^{-1}(A) = f(A)$. Recall
also that branched covering maps are open and hence confluent.

\begin{thm}\label{dyn1}
  Suppose that $f:\C\to \C$ is a branched covering map and that $X$
  is a fully invariant unshielded compactum. Then there exists a
  branched covering map $g:\C\to \C$ such that $M_X\circ f=g\circ
  M_X$ and hence $X_{FS}=M_X(X)$ is fully invariant under $g$.
\end{thm}

\begin{proof}
  By Lemma~\ref{finmocol}, $m_X$ sends limit continua into
  $\fs$-classes. Lemma~\ref{lem:induced} and
  Theorem~\ref{finmod1}, in which the extension $M_X$ of $m_X$ onto
  $\C$ is described, imply that $g = M_X \circ f \circ M_X^{-1}$ is
  well-defined. Suppose that we show that $\fs$-classes map \emph{onto}
  image classes. Then, since $f$ is open, it will follow that $g$ is open too.
  Moreover, let us show that then $g$ is light. Indeed, if $x\in M_X(X)$ then $M_X^{-1}(x)$ is
  an $\fs$-class in $X$. Since we assume that classes map onto
  classes and the map $f$ is finite-to-one, we see that each
  component of $f^{-1}(M_X^{-1}(x))$ is an $\fs$-class. Hence $g$ is
  finite-to-one. Since by the Stoilow theorem \cite{sto56} all open
  finite-to-one maps of the plane are branched covering maps, $g$ is
  a branched covering map as desired.

  To see that the image of a $\fs$-class is again a $\fs$-class, we
  show that $\sim_\alpha$-classes map onto the union of
  $\sim_\alpha$-classes for every ordinal $\alpha$, where
  $\sim_\alpha$ was defined in Remark~\ref{rem:transfinite} with
  $\mathcal A$ being the set of limit continua.  Then, when $\alpha =
  \Omega$, we see that $\fs$-classes map both into and over other
  $\fs$-classes.

  Let us first show that $\sim_0$-classes map over other
  $\sim_0$-classes. Indeed, let $f(x)$ and $y$ belong to the same
  $\sim_0$-class. Then there exist finitely many limit continua $C_1 =
  \lim_{i \to \infty} C_i^1$, \ldots, $C_n = \lim_{i \to \infty} C_i^n$
  forming a chain joining $f(x)$ and $y$ (i.e., so that $f(x) \in C_1$,
  $y \in C_n$, and $C_j \cap C_{j+1} \neq \emptyset$ for any $1 \le j <
  n$).  Since $f$ is an open map, there exists a convergent sequence
  $(D_i^1)_{i=1}^\infty \to D_1$ of continua such that $f(D_i^1) =
  C_i^1$ for each $i$ and $D_1$ is a limit continuum which contains
  $x$. By continuity, $f(D_1) = C_1$, so $D_1$ contains the preimage of
  a point in $C_2$.  We can now inductively find limit continua $D_2,
  \ldots, D_n$ mapping onto $C_2, \ldots, C_n$ and forming a chain from
  $x$ to a preimage of $y$.  Therefore, the $\sim_0$-class of $f(x)$ is
  contained in the image of the $\sim_0$-class of $x$.

  Suppose now by induction that we have proven the claim for all
  ordinals less than $\alpha$, and let $f(x) \sim_\alpha y$. If
  $\alpha$ has an immediate predecessor $\beta$ (the other case is left
  as an easy exercise for the reader), there are finitely many
  sequences of $\sim_\beta$-classes $(K_i^1)_{i=1}^\infty$, \ldots,
  $(K_i^n)_{i=1}^\infty$ which converge to a chain of continua joining
  $f(x)$ and $y$.  By the inductive hypothesis, if $f(z) \in K_i^j$
  then the $\sim_\beta$-class of $z$ maps over $K_i^j$.  One can
  therefore find, due to the openness of $f$, a convergent sequence
  $(L_i^1)_{i=1}^\infty \to L_1$ of $\sim_\beta$-classes such that
  $f(L_i^1) \supset K_i^1$ and $x \in L_1$.  Note by continuity that
  $f(L_1) \supset K_1$.  Proceeding as in the previous paragraph, we
  find similar limits $L_2$, \ldots, $L_n$ forming a chain of continua
  which joins $x$ to a preimage of $y$.  We therefore see that the
  image of a $\sim_\alpha$-class is a union of $\sim_\alpha$-classes,
  and the proof is complete.
\end{proof}

Sometimes in the situation of Theorem~\ref{dyn1} a naive but natural approach
to the problem of constructing the finest \FS{} model can be used.

\begin{defn}\label{naive}
  By Theorem~\ref{finmod1} for each component $Y$ of $X$, the finest
  equivalence relation on $Y$ is $\fs_Y$. Consider the equivalence
  relation $\sim_{\blacktriangle, X}$ defined as follows: $x\sim_{\blacktriangle, X} y$ if and
  only if $x$ and $y$ belong to the same component $Y$ of $X$ and
  $x\fs_Y y$.
\end{defn}

If $X$ is given or non-essential, we will simply  write
$\sim_\blacktriangle$ or $\fs$.  It is natural to find out if
$\sim_{\blacktriangle, X}$ coincides with $\fs_X$. Simple examples
show that in general it is not true.

\begin{exm}[A map on a compactum $X$ with $\sim_{\blacktriangle, X}$ not coinciding with
  $\fs_X$]\label{simnotfs}
  Define a map $f:\C\to \C$ as follows. Take a map from the real
  quadratic family $g_b(x)=bx(1-x)$ with $b>4$. It is well known that
  then there exists a forward invariant Cantor set $A\subset [0, 1]
  \setminus \{0\}$ on which the map $g_b$ acts as a full 2-shift. We
  define a map on the set $X = A \times [-1, 1]$ as $f(x,y) =
  \left(g_b(x), y\right)$.  
  Evidently, $f$ can be extended to a branched covering two-to-one map
  $f:\C\to \C$, however for brevity we will not give its full
  description here.

  Observe that $X$ is a fully invariant set.  The equivalence relation
  $\fs_X$ collapses $X$ to a Cantor set, though all $\sim_{\blacktriangle,
  X}$-classes are points.  Thus, in this case $\sim_{\blacktriangle, X}\ne \fs_X$.
\end{exm}

Example~\ref{simnotfs} shows that in some cases
$\sim_{\blacktriangle, X}$ and $\fs_X$ are distinct.  Moreover, it
also shows the mechanism of how this distinction occurs. However,
the definition immediately implies that $\sim_{\blacktriangle,X}$ is
finer than $\fs_X$.



It turns out that the abberation $\fs_X \neq \sim_{\blacktriangle,
X}$ is impossible for polynomial maps. To show that we need some
definitions. A point $x\in X$ of a planar compactum $X$ is called
\textbf{accessible (from $U_\iy(X)$)} if there is a curve $Q\subset
U_\iy(X)$ with one endpoint at $x$ (then one says that $Q$
\textbf{lands} at $x$ and that $x$ is \textbf{accessible by $Q$}).
We also need the definition of impression of an angle.  For a
continuum $X \subset \C$, let $\psi:\C \setminus \Disk \to U_\iy(K)$
denote the unique conformal isomorphism with real derivative at
$\infty$.  For an angle $\alpha \in \ucirc$, we define the
\textbf{impression} of the external ray at angle $\alpha$ as
\[ \imp(\alpha) = \{ \lim \psi(z_i) \, : \, z_i \to \alpha \text{ from
  within $\Disk$}\}.\]

\begin{thm}\label{naifs}
  Let $P:\C\to \C$ be a polynomial. Then the equivalence relations
  $\sim_{\blacktriangle, J_P}$ and $\fs_{J_P}$ coincide.
\end{thm}

\begin{proof}
  A recent result by \cite{ks06, qy06} states that \emph{all
    non-preperiodic components of a polynomial Julia set are points}.
  Consider $J_P$ as the given compact set. Then
  it is enough to show that if $x, y\in J_P$ and $x\fs y$, then
  $x\sim_\blacktriangle y$. By Definition~\ref{rx}, it suffices to show that a limit
  continuum in $J_P$ is contained in a $\sim_\blacktriangle$-class. Let $C\subset
  J_P$ be a limit continuum and $C_i\to C$ is a sequence of
  subcontinua of $J_P$ which converge to $C$. Denote by $T_i$ the
  component of $J_P$ containing $C_i$ and by $T$ the component of
  $J_P$ containing $C$.

If infinitely many $C_i$'s are contained in $T$, then by
Definition~\ref{rx} $C$ is contained in a
$\sim_\blacktriangle$-class and we are done. Suppose that there are
only finitely many $C_i$'s in $T$. Then we may assume that a
sequence of pairwise distinct components $T_i$ converges to a limit
continuum $C'\subset T$ where $C\subset C'$, and we need to show
that $C'$ is contained in one $\sim_\blacktriangle$-class. To do so,
we consider two cases.

First, assume that $T$ is periodic of period $m$. Then it is
well-known that $P^m|_T$ is a so-called \textbf{polynomial-like} map
(see \cite{dh85}) for which $T$ plays the role of its filled-in
Julia set. That is, there exist two simply connected neighborhoods
$U\subset V$ of $T$ such that $P^m:U\to V$ is a branched covering
map and there exist a polynomial $f$ with connected Julia set $J_f$
and two neighborhoods $U'\subset V'$ of $J_f$ such that $P^m|_U$ is
(quasi-conformally) conjugate to $f|_{U'}$ by a homeomorphism $\ph$
and $\ph(T)=K_f$ where $K_f$ is the \textbf{filled-in Julia set of
$f$} (i.e., the topological hull of $J_f$). We will use
a conformal map $\psi:\C\sm K_f\to \C\sm \ol{\disk}$ which
conjugates $f|_{\C\sm K_f}$ and $z^d|_{\C\sm \ol{\disk}}$.

We claim that there exists an angle $\al$ such that $C'\subset
\ph^{-1}(\imp(\al))$. We will consider continua $\psi\circ
\ph(T_i)=T'_i\subset \C\sm \ol{\disk}$ and will show that they
converge to a unique point in $\uc$. Indeed, otherwise we may assume
that they converge to an non-degenerate arc $I\subset \uc$. For any
$t\in \uc$ let $R_t\subset \C\sm \disk$ be the half-line from $t$ to
infinity, orthogonal to $\uc$ at $t$. Choose $\be\in I$ such that
the $R'=\psi^{-1}(R_\be)$ is a curve in $\C\sm K_f$ landing at a
point $b\in T$. We may assume that $\be$ is not an endpoint of $I$.

We need Theorem~2 of \cite{lp96} which states that if $x\in T$ is an
accessible point from $\C\sm T$ by a curve $l$, then $x$ is
accessible from $\C\sm J_P$ by a curve $R$ which is homotopic to $l$
among all curves in $\C\sm T$ landing at $x$. By this result we can
find a curve $L\subset \C\sm K_f$ which lands at $b$ and is disjoint
from $\ph(J_P)$. Then the curve $\psi(L)\subset \C\sm \ol{\disk}$
lands at $\be$ while being disjoint from all sets $T'_i$ which
clearly contradicts the assumption that these sets converge to the
arc $I$.

Thus, we may assume that $T'_i\to \al\in \uc$ which, by the
definition of impression, implies that $T_i$ converge into the set
$\ph^{-1}(\imp(\al))$ and so $C'\subset \ph^{-1}(\imp(\al))$. Now,
Lemma~16 of \cite{bco08} states that any monotone map of a connected
Julia set onto a locally connected continuum collapses impressions
of external rays to points. Hence the set $\imp(\al)$ is contained
in one $\sim_{\blacktriangle, J_f}$-class which implies (after we
apply the homeomorphism $\ph^{-1}$ to this) that the set
$\ph^{-1}(\imp(\al))$ is contained in one $\sim_{\blacktriangle,
T}$-class as desired. This completes the consideration of the case
of a periodic $T$.

Now, suppose that $T$ is not periodic. Then by \cite{ks06, qy06} $T$
is preperiodic and we can choose $n>0$ such that $P^n(T)$ is a
periodic component of $J_P$. Since by the above all limit continua
in $P^n(T)$ are contained in $\sim_\blacktriangle$-classes, it is
easy to use pullbacks to see that all limit continua in $T$ are
contained in $\sim_\blacktriangle$-classes too. This completes the
proof.
\end{proof}

The following two corollaries easily follow.

\begin{cor}\label{nondeg}
  The finest \FS{} model of $J_P$ 
  has at least one non-degenerate component if and only if there
  exists a periodic component of $J_P$ which has a non-degenerate
  \FS{} model.
\end{cor}

Observe that a dynamical criterion for a connected Julia set to have a
non-degenerate \FS{} model is obtained in \cite{bco08}.

\begin{cor}\label{fsjset}
The set $J_P$ is \FS{} if and only if all periodic components of $J_P$ are locally connected.
\end{cor}

\noindent \textbf{Acknowledgements.} We thank the referee for useful
and thoughtful remarks.

\providecommand{\bysame}{\leavevmode\hbox to3em{\hrulefill}\thinspace}
\providecommand{\MR}{\relax\ifhmode\unskip\space\fi MR }
\providecommand{\MRhref}[2]{%
  \href{http://www.ams.org/mathscinet-getitem?mr=#1}{#2}
}
\providecommand{\href}[2]{#2}


\begin{thebibliography}{BMO07}

\bibitem[BCO08]{bco08} A. Blokh, C. Curry, L. Oversteegen,
\emph{Locally connected models for Julia sets}, \textbf{226} (2011),
1621--1661.

\bibitem[BMO07]{bmo07}
Alexander Blokh, Micha{\l} Misiurewicz, and Lex Oversteegen, \emph{Planar
  finitely {S}uslinian compacta}, Proc. Amer. Math. Soc. \textbf{135} (2007),
  no.~11, 3755--3764. 

\bibitem[BO04]{bo04}
Alexander Blokh and Lex Oversteegen, \emph{Backward stability for polynomial
  maps with locally connected {J}ulia sets}, Trans. Amer. Math. Soc.
  \textbf{356} (2004), no.~1, 119--133. 

\bibitem[DH85]{dh85} A. Douady, J. H. Hubbard, \emph{On the dynamics
      of polynomial-like mappings}, Ann. Sci. \'Ecole Norm. Sup. (4)
      \textbf{18} (1985), 287--343.

\bibitem[FS67]{Swingle:Core}
R.~W. FitzGerald and P.~M. Swingle, \emph{Core decomposition of continua},
  Fund. Math. \textbf{61} (1967), 33--50. 

\bibitem[GM93]{gm93} L. Goldberg, J. Milnor, \emph{
Fixed points of polynomial maps. II. Fixed point portraits},
Ann. Sci. \'Ecole Norm. Sup. (4) \textbf{26} (1993), no. 1, 51--98.

\bibitem[Kiw04]{kiw04}
Jan Kiwi, \emph{{$\mathbb R$}eal laminations and the topological dynamics of
  complex polynomials}, Adv. Math. \textbf{184} (2004), no.~2, 207--267.

\bibitem[KS06]{ks06} O. Kozlovski, S. van Strien, \emph{Local connectivity
    and quasi-conformal rigidity of non-renormalizable polynomials}, Proc.
    Lond. Math. Soc. (3), \textbf{99} (2009), 275--296.

\bibitem[Kur66]{kur66} K. Kuratowski, Topology, Volume 1, Academic Press, New York, 1966.

\bibitem[LP96]{lp96} G. Levin, F. Przytycki, \emph{External rays to periodic
    points,} Israel J. Math. \textbf{94} (1996), 29--57.

\bibitem[Moo62]{moo62} R. L. Moore, \emph{Foundations of point set theory.
    Revised edition}, AMS Colloquium Publications \textbf{13} (1962), AMS,
    Providence, R.I.

\bibitem[Mun00]{mun00} J. Munkres, \emph{Topology}, 2nd edition,
    Prentice Hall (2000).

\bibitem[Nad92]{Nadler:ContTh}
Sam~B. Nadler, Jr., \emph{Continuum theory}, Monographs and Textbooks in Pure
  and Applied Mathematics, vol. 158, Marcel Dekker Inc., New York (1992). 

\bibitem[QY06]{qy06} W. Qiu, Y. Yin, \emph{Proof of the Branner-Hubbard
conjecture on Cantor Julia sets,} Sci. China Ser. A, \textbf{52}  (2009), no.
1, 45--65

\bibitem[Sto56]{sto56} S. Stoilow, \emph{Le\c cons sur les principes topologique
de la theorie des fonctions analytique}, 2nd edition, Paris (1956).

\end{thebibliography}
\end{document}